\documentclass[11pt]{amsart}

\usepackage{amsmath, amstext, amsgen, amsbsy, amsopn, amsfonts, amssymb, graphicx,psfrag}
\usepackage[normalem]{ulem}
\usepackage{pdfsync}
\usepackage{epstopdf}
\usepackage{color}
\usepackage[thinlines]{easytable}

\newtheorem{theorem}{Theorem}

\newtheorem{lemma}[theorem]{Lemma}
\newtheorem{corollary}[theorem]{Corollary}

\begin{document}

%\markboth{Mellor}{Finite involutory quandles of two-bridge links with an axis}

\title{Finite involutory quandles of two-bridge links with an axis}

\author{Blake Mellor}
\address{Loyola Marymount University, 1 LMU Drive, Los Angeles, CA 90045}
\email{blake.mellor@lmu.edu}

\date{}
\maketitle
\begin{abstract} To better understand the fundamental quandle of a knot or link, it can be useful to look at finite quotients of the quandle.  One such quotient is the $n$-quandle (or, when $n=2$, the {\em involutory} quandle).  Hoste and Shanahan \cite{HS2} gave a complete list of the links which have finite $n$-quandles; it remained to give explicit descriptions of these quandles. This has been done for several cases in \cite{CHMS} and \cite{HS1}; in the current work we continue this project and explicitly describe the Cayley graphs for the finite involutory quandles of two-bridge links with an axis.  \end{abstract}

\section{Introduction}
Every oriented knot and link $L$ has a fundamental quandle $Q(L)$. For tame knots, Joyce \cite{JO2, JO} showed that the fundamental quandle is a complete invariant (up to a change of orientation).  Unfortunately, classifying quandles is no easier than classifying knots - in particular, except for the unknot and Hopf link, the fundamental quandle is always infinite.  This motivates us to look at quotients of the quandle, such as the $n$-quandle $Q_n(L)$.  Hoste and Shanahan \cite{HS2} proved that the $n$-quandle $Q_n(L)$ is finite if and only if the $n$-fold cyclic branched cover of $S^3$ branched over $L$, has finite fundamental group. Using this result, together with Dunbar's \cite{DU} classification of all geometric, non-hyperbolic 3-orbifolds, they were able to give a complete list of all  knots and links in $S^3$ with finite $n$-quandle for some $n$ \cite{HS2}; see Table \ref{linktable}.  Hoste and Shanahan \cite{HS1} described the finite $n$-quandles associated to Montesinos links, and Crans, Hoste, Shanahan and the author \cite{CHMS} described those associated to two-bridge links, torus links and torus links with an axis.  This leaves four infinite families of links which have finite $n$-quandles - specifically, finite 2-quandles (also known as {\em involutory quandles}).  In the current paper we consider the involutory quandles of two-bridge links with an axis, denoted $L(k,p/q) \cup C$, as shown in Figure \ref{F:L(k,p,q)UC}.  We will give an explicit description of the Cayley graphs for these involutory quandles (Theorem \ref{T:Cayley}), and show that $\vert Q_2(L)\vert = 2q(\vert kq-p\vert +1)$ (Corollary \ref{C:cardinality}).

\begin{table}[htbp]
$$
\begin{array}{ccc}
\includegraphics[width=1.25in,trim=0 0 0 0,clip]{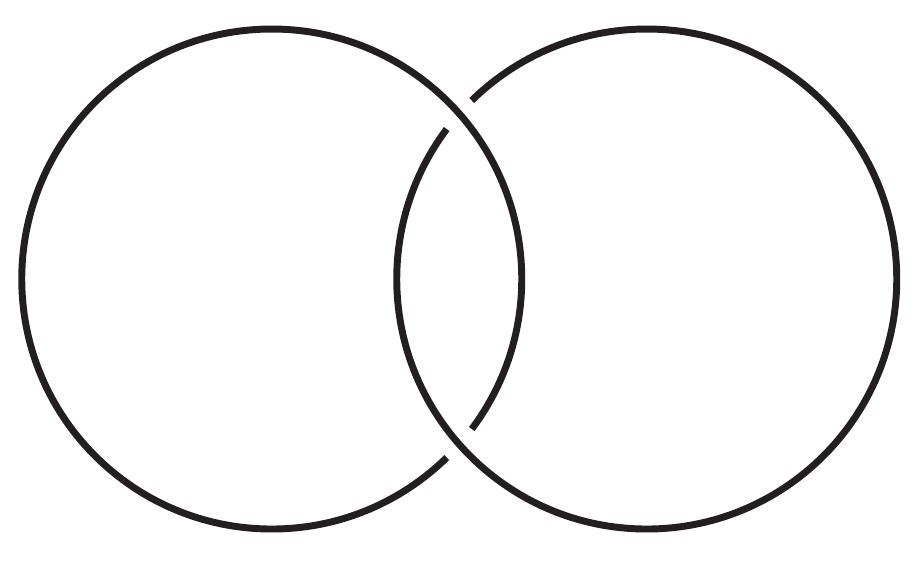}   & \includegraphics[width=1.25in,trim=0 0 0 0,clip]{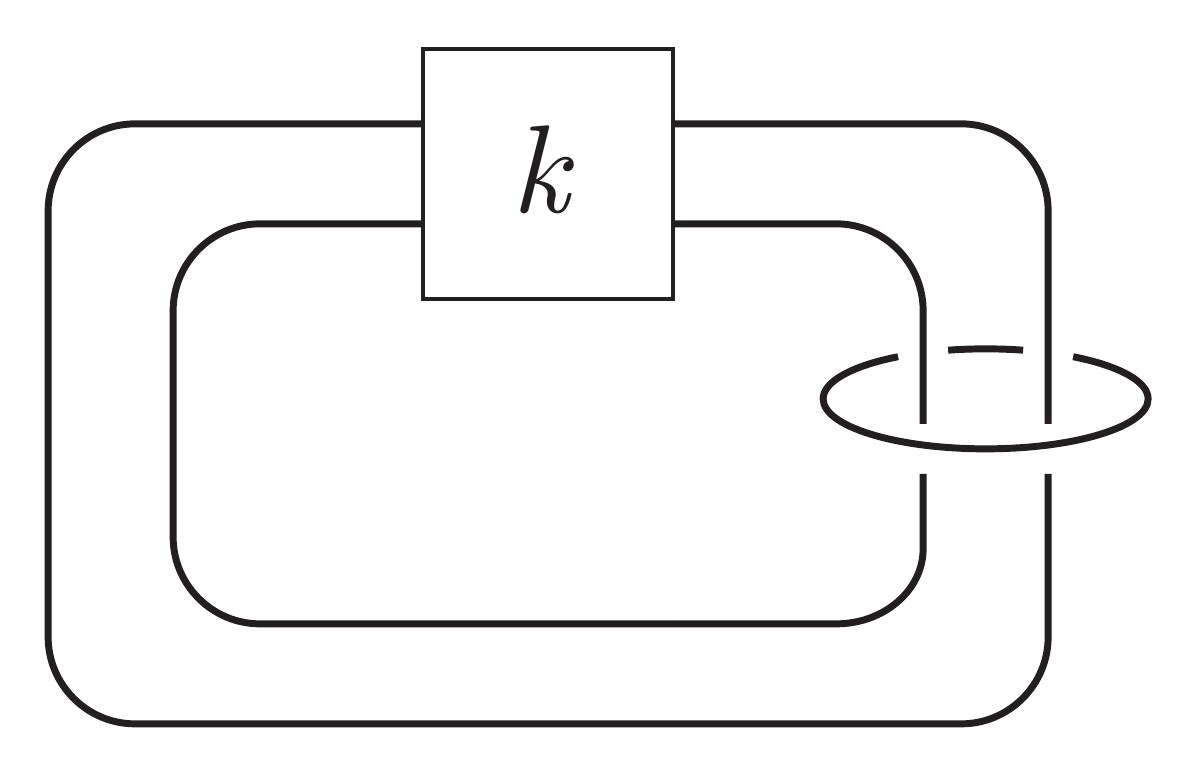}  & \includegraphics[width=1.0in,trim=0 0 0 0,clip]{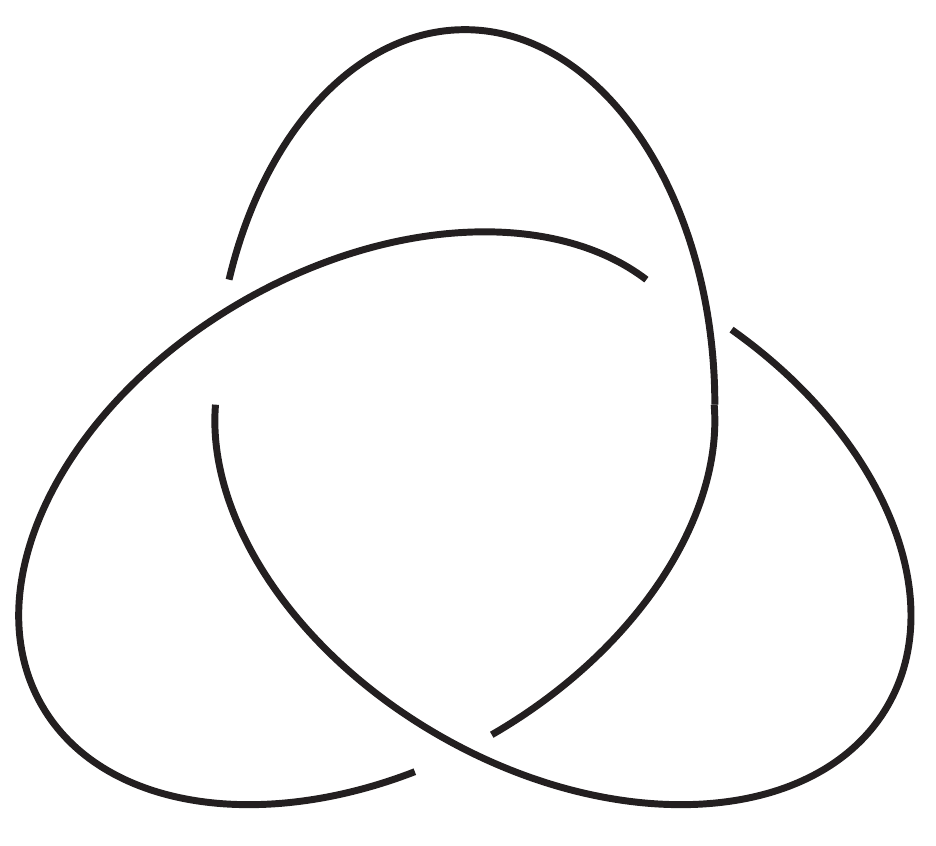} \\
\scriptstyle n > 1 & \scriptstyle k \neq 0,\  n=2& \scriptstyle n=3, 4, 5 \\
\\
\includegraphics[width=1.0in,trim=0 0 0 0,clip]{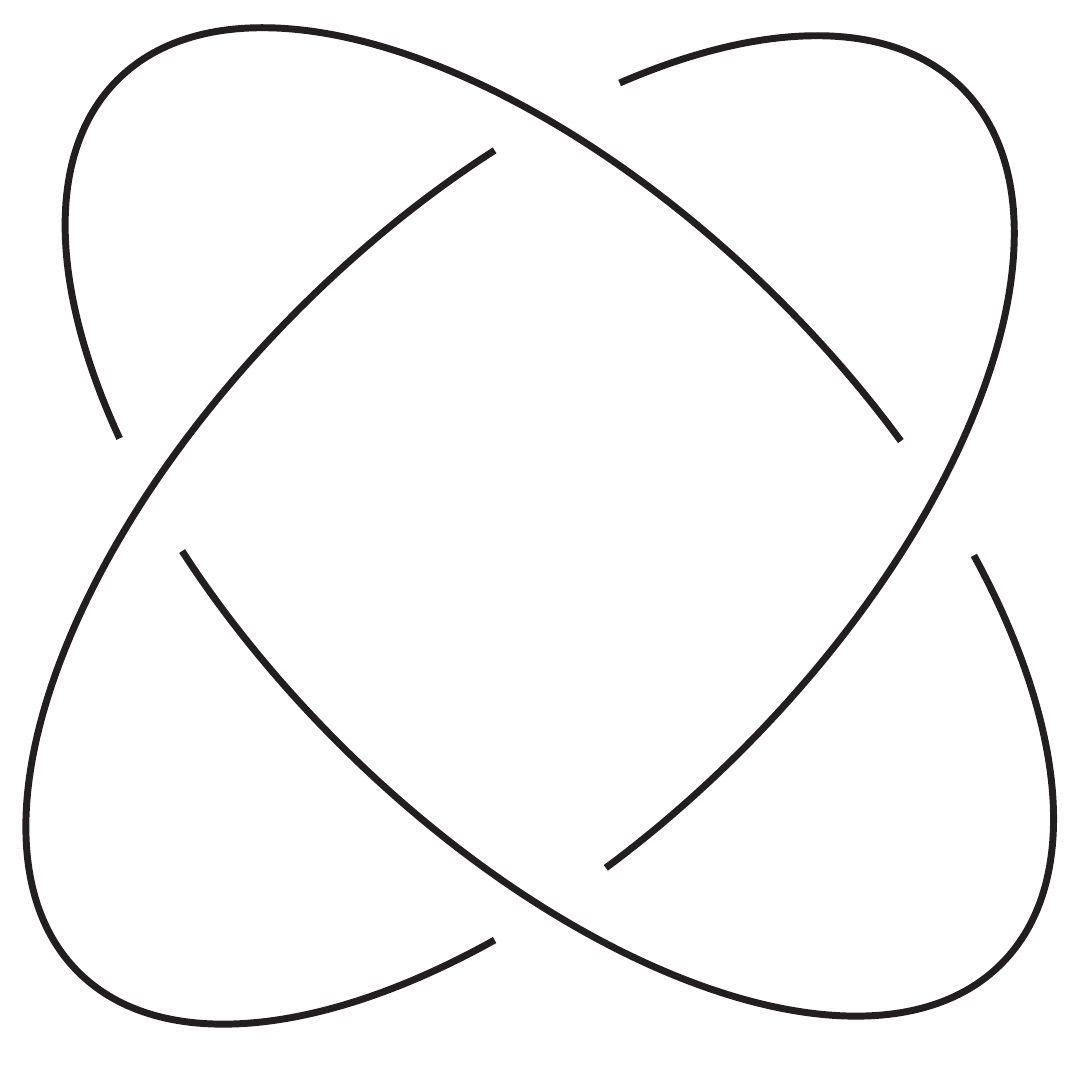}  & \includegraphics[width=1.0in,trim=0 0 0 0,clip]{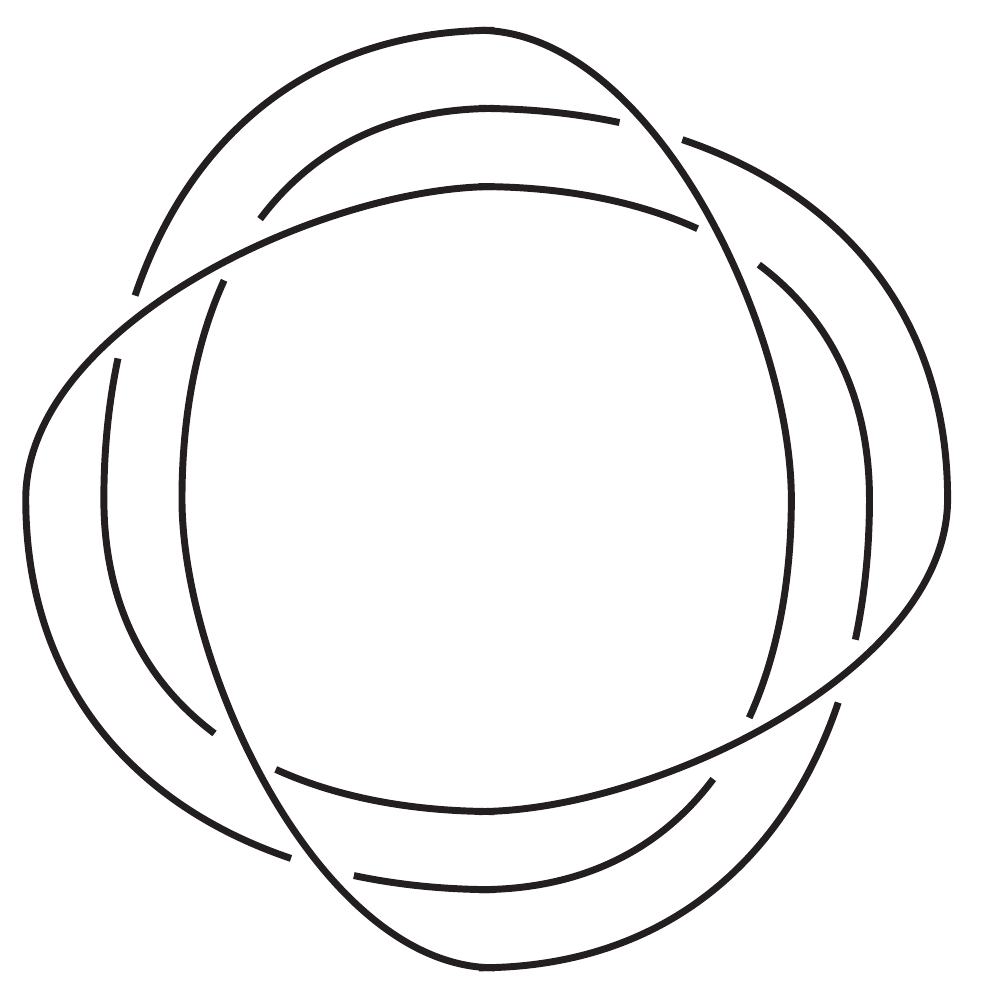}  & \includegraphics[width=1.0in,trim=0 0 0 0,clip]{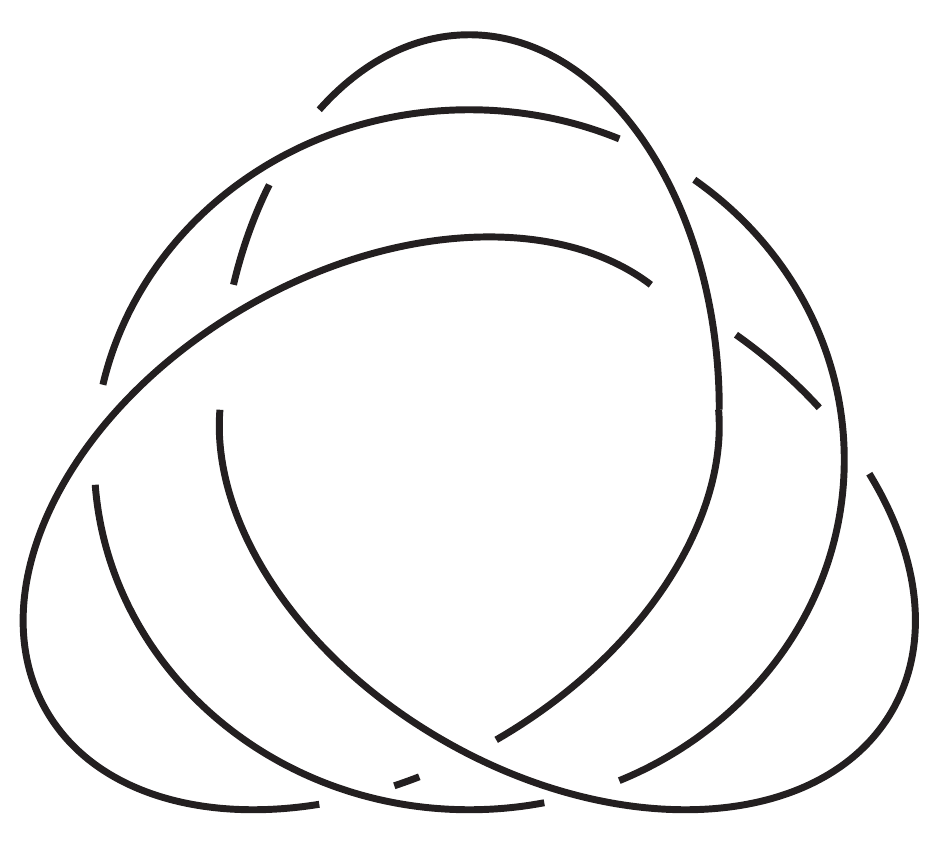} \\
\scriptstyle n =3 & \scriptstyle n=2& \scriptstyle n=2 \\
\\
\includegraphics[width=1.0in,trim=0 0 0 0,clip]{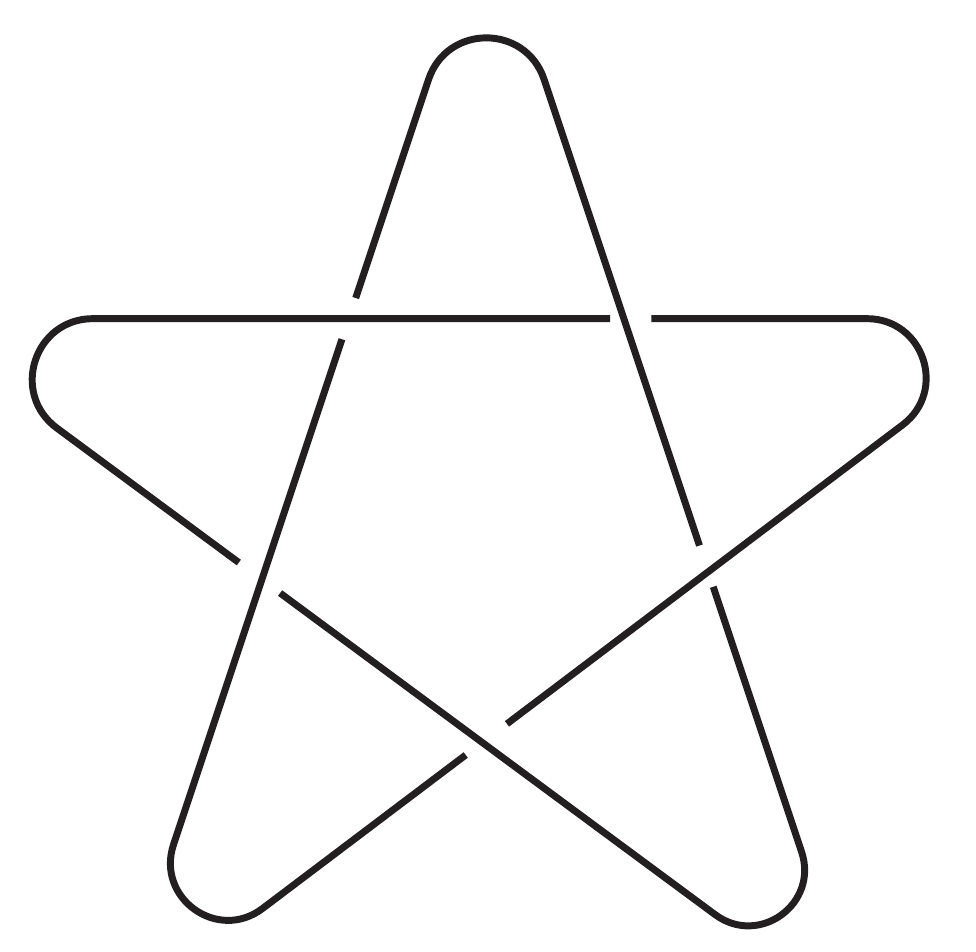}  & \includegraphics[width=1.0in,trim=0 0 0 0,clip]{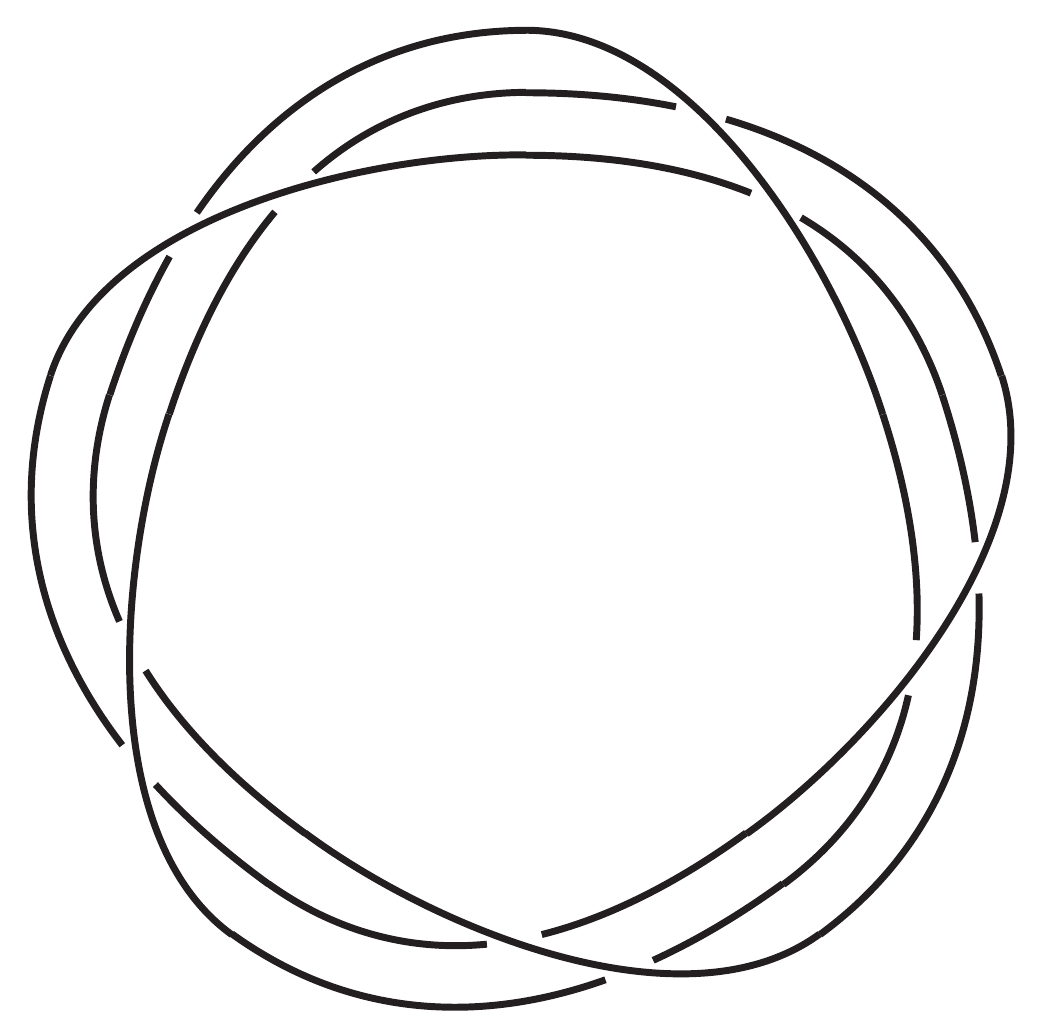}  & \includegraphics[width=1.25in,trim=0 0 0 0,clip]{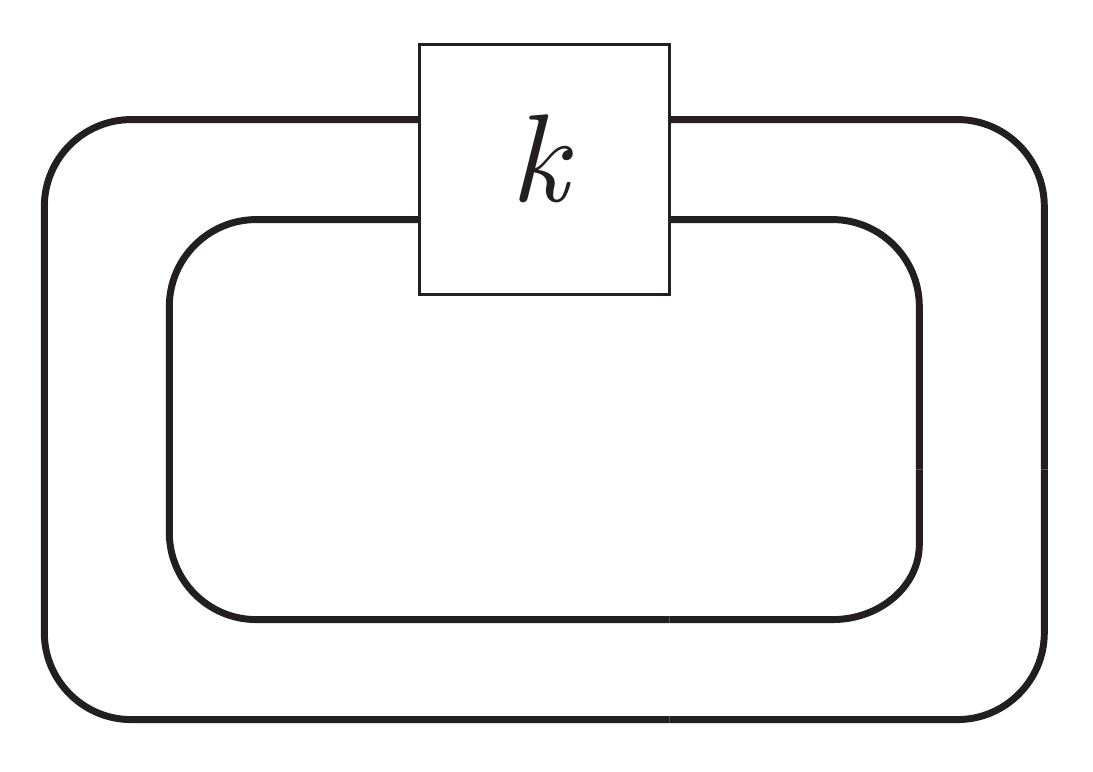} \\
\scriptstyle n =3 & \scriptstyle n=2& \scriptstyle k\neq0,\ n=2 \\
\\
\includegraphics[width=1.25in,trim=0 0 0 0,clip]{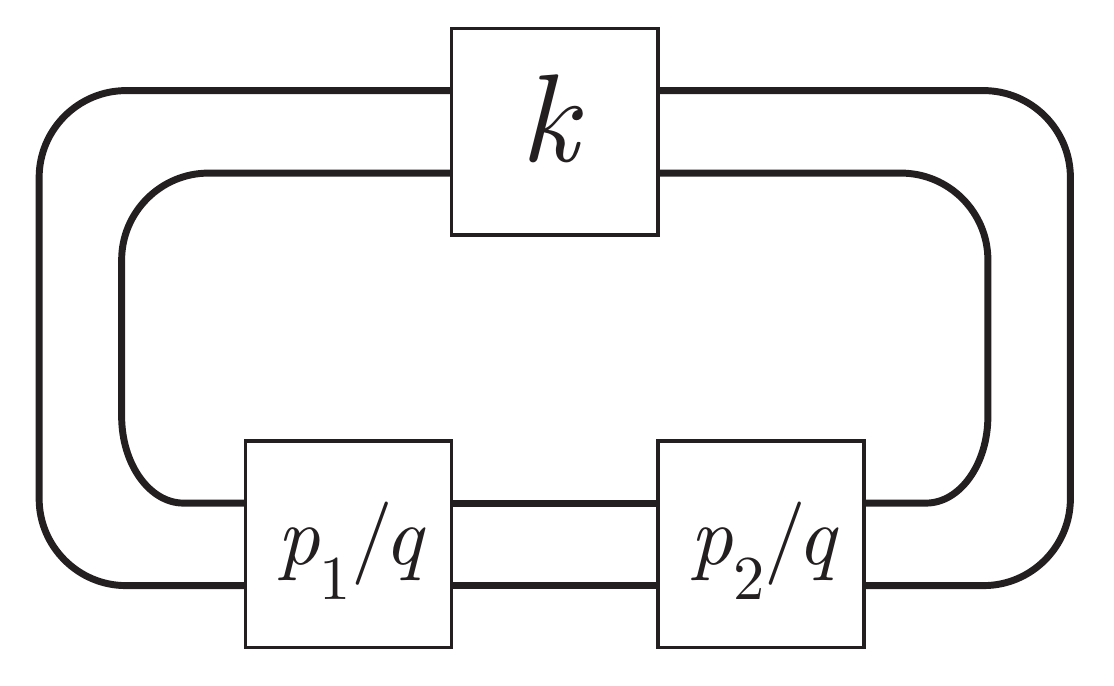}  & \includegraphics[width=1.15in,trim=0 5pt 0 0,clip]{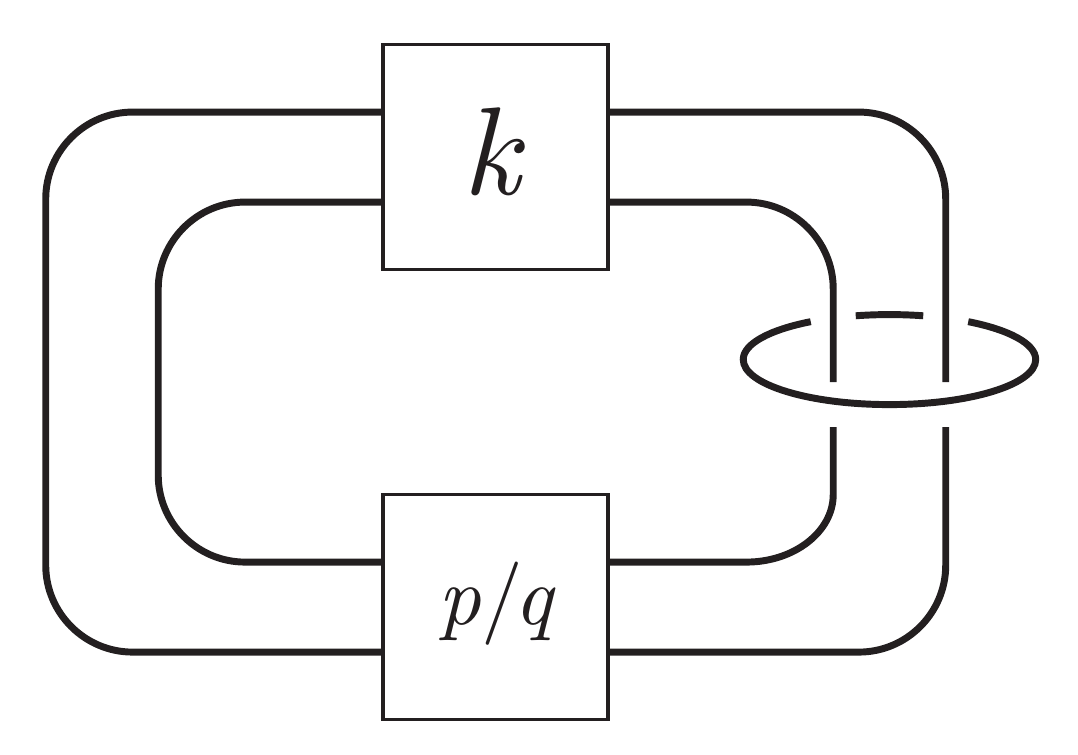} &  \includegraphics[width=1.65in,trim=0 0 0 0,clip]{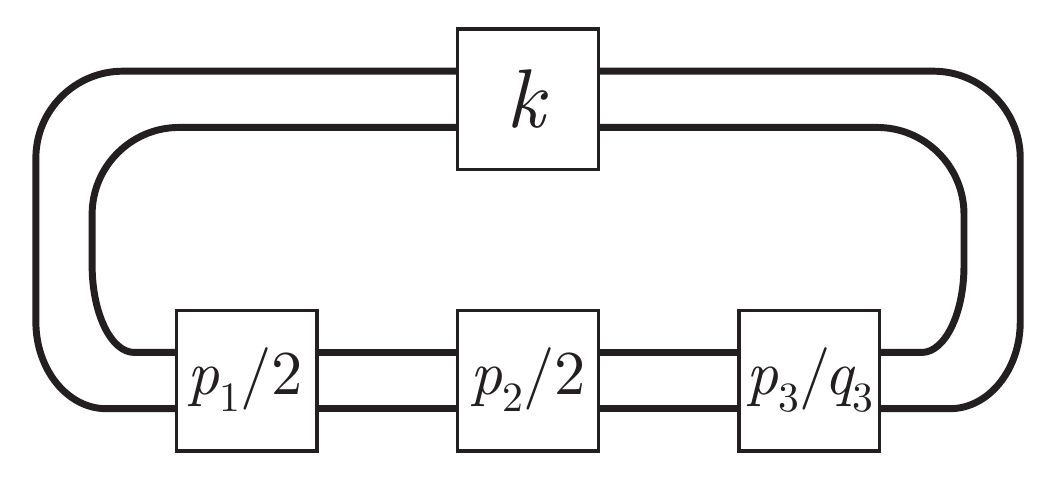} \\
\scriptstyle k+p_1/q+p_2/q \neq 0,\ n =2  &\scriptstyle n=2& \scriptstyle k+p_1/2+p_2/2+p_3/q_3 \neq 0,\ n =2 \\
\\
\includegraphics[width=1.65in,trim=0pt 0pt 0pt 0pt,clip]{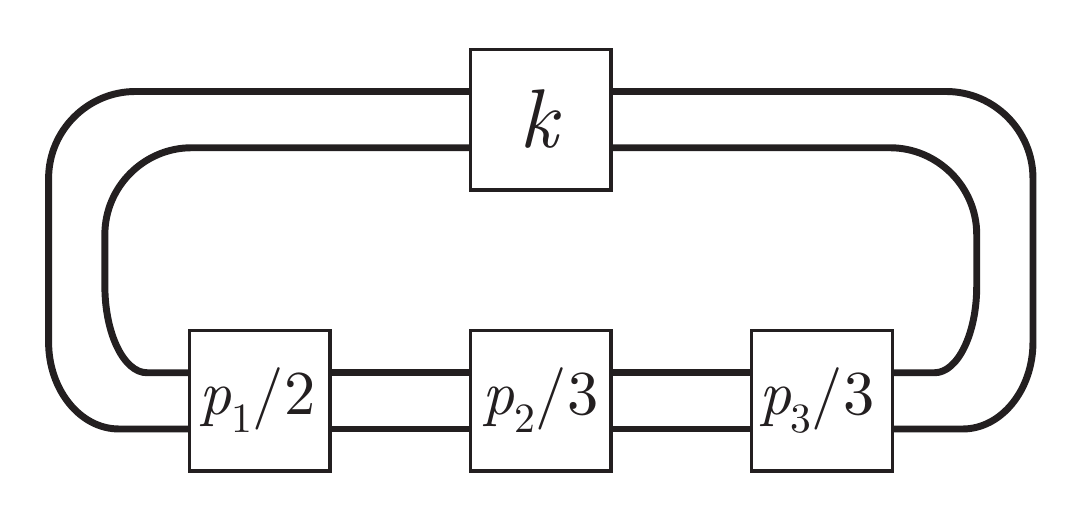}  & \includegraphics[width=1.65in,trim=0pt 0pt 0pt 0pt,clip]{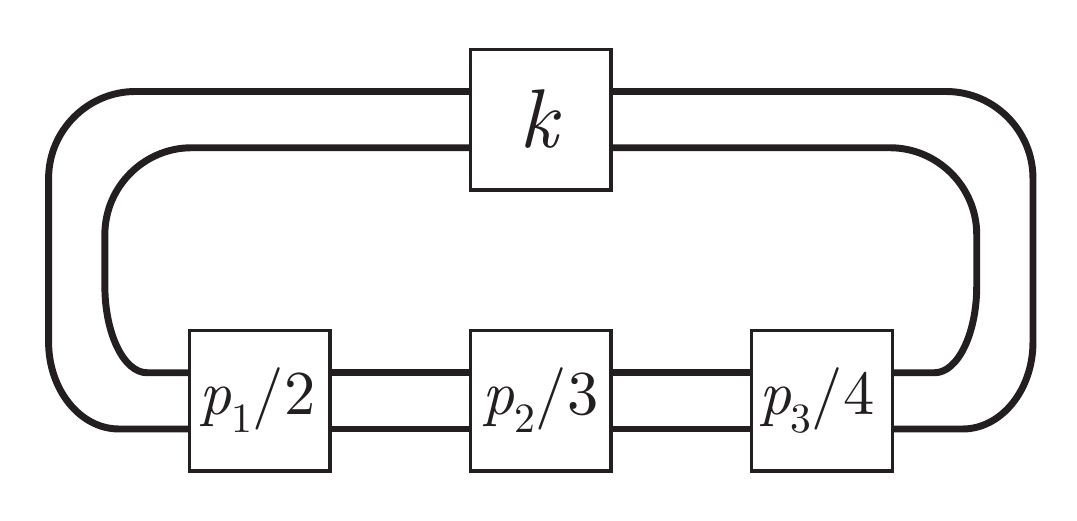}  & \includegraphics[width=1.65in,trim=0pt 0pt 0pt 0pt,clip]{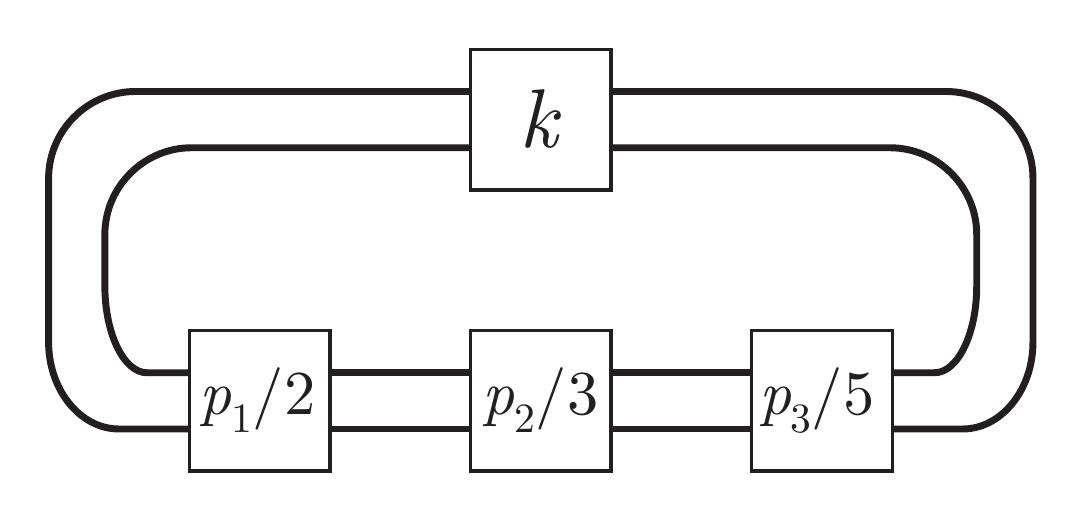}\\
\scriptstyle k+p_1/2+p_2/3+p_3/3 \neq 0,\ n =2  & \scriptstyle k+p_1/2+p_2/3+p_3/4 \neq 0,\ n =2  & \scriptstyle k+p_1/2+p_2/3+p_3/5 \neq 0,\ n =2
\end{array}
$$
\caption{Links $L \in \mathbb{S}^3$ with finite $Q_n(L)$}
\label{linktable}
\end{table}

\begin{figure}[htbp]
\centerline{\scalebox{.5}{\includegraphics{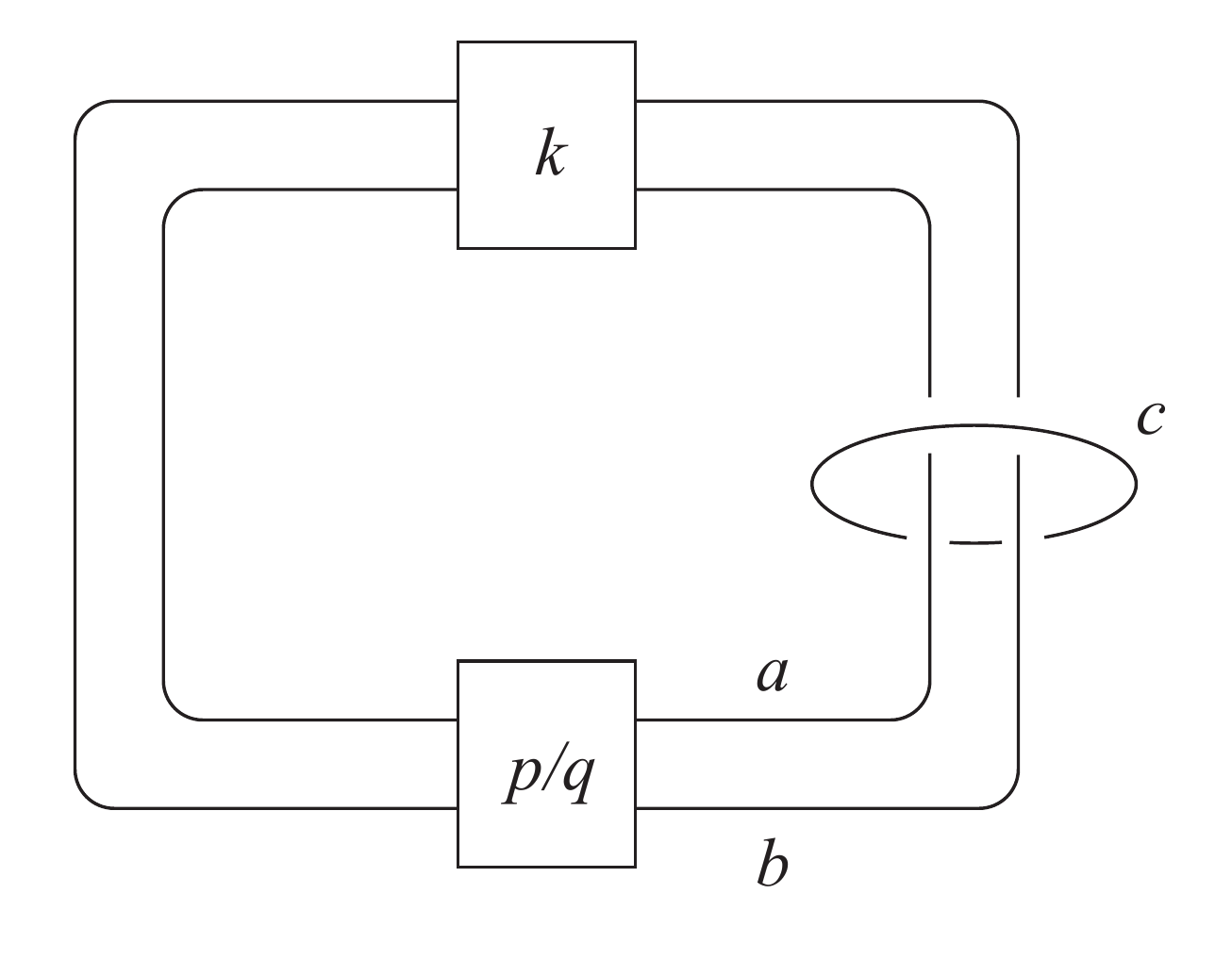}}}
\caption{The two bridge link $L(k,p/q)$ with axis $C$.}
\label{F:L(k,p,q)UC}
\end{figure}

\section{Link quandles} \label{S:quandles}

We begin with a review of the definition of the fundamental quandle of a link and its associated $n$-quandles. We refer the reader to \cite{FR}, \cite{JO2}, \cite{JO}, and \cite{WI} for more detailed information.

A {\it  quandle} is a set $Q$ equipped with two binary operations $\rhd$ and $\rhd^{-1}$ that satisfy the following three axioms:
\begin{itemize}
\item[\bf A1.] $x \rhd x =x$ for all $x \in Q$.
\item[\bf A2.] $(x \rhd y) \rhd^{-1} y = x = (x \rhd^{-1} y) \rhd y$ for all $x, y \in Q$.
\item[\bf A3.] $(x \rhd y) \rhd z = (x \rhd z) \rhd (y \rhd z)$ for all $x,y,z \in Q$.
\end{itemize}

%Each element $x\in Q$ defines a map $S_x:Q \to Q$ by $S_x(y)=y \rhd x$. The axiom A2 implies that each $S_x$ is a bijection and the axiom A3 implies that each $S_x$ is a quandle homomorphism, and therefore an automorphism. We call $S_x$ the {\it point symmetry at $x$}.

An {\em involutory} quandle (or $2$-quandle) is the quotient $Q_2$ of $Q$ found by letting $\rhd^{-1} = \rhd$.  In this case, axiom A2 becomes $(x \rhd y) \rhd y = x$ for all $x, y \in Q_2$.  In this paper, we will primarily work with involutory quandles. %Note that a quandle is involutory if and only if every point symmetry is an involution. 

It is important to note that the operation $\rhd$ is, in general, not associative. To clarify the distinction between $(x \rhd y) \rhd z$ and $x \rhd (y \rhd z)$, we adopt the exponential notation introduced by  Fenn and Rourke in \cite{FR} and denote $x \rhd y$ as $x^y$ and $x \rhd^{-1} y$ as $x^{\bar y}$. With this notation, $x^{yz}$ will be taken to mean $(x^y)^z=(x \rhd y)\rhd z$ whereas $x^{y^z}$ will mean $x\rhd (y \rhd z)$. 

The following useful lemma from \cite{FR} describes how to re-associate a product in a quandle given by a presentation. 

\begin{lemma} \label{leftassoc}
If $a^u$ and $b^v$ are elements of a quandle, then
$$\left(a^u \right)^{\left(b^v \right)}=a^{u \bar v b v} \ \ \ \ \mbox{and}\ \ \ \ \left(a^u \right)^{\overline{\left(b^v \right)}}=a^{u \bar v \bar b v}.$$
\end{lemma}

Using Lemma~\ref{leftassoc}, elements in a quandle given by a presentation $\langle S \mid R \rangle$ can be represented as equivalence classes of expressions of the form $a^w$ where $a$ is a generator in $S$ and $w$ is a word in the free group on $S$ (with $\bar x$ representing the inverse of $x$).

In an involutory quandle, each element of $S$ has order 2 in the free group, so $\bar x = x$ for each $x \in S$.  If $w = x_1x_2\dots x_k$ is a word in the free group on $S$, then $\bar w = {\bar x_k} \dots {\bar x_2} \bar x_1 = x_k\dots x_2x_1$.

If $L$ is an oriented knot or link in $S^3$, then a presentation of its fundamental quandle, $Q(L)$, can be derived from a regular diagram $D$ of $L$ by a process similar to the Wirtinger algorithm \cite{JO}. We assign a quandle generator $x_1, x_2, \dots , x_n$ to each arc of $D$, then at each crossing introduce the relation $x_i=x_k^{ x_j}$ as shown in Figure~\ref{crossing}. It is easy to check that the three Reidemeister moves do not change the quandle given by this presentation so that the quandle is indeed an invariant of the oriented link.

\begin{figure}[h]
$$\includegraphics[height=1.25in]{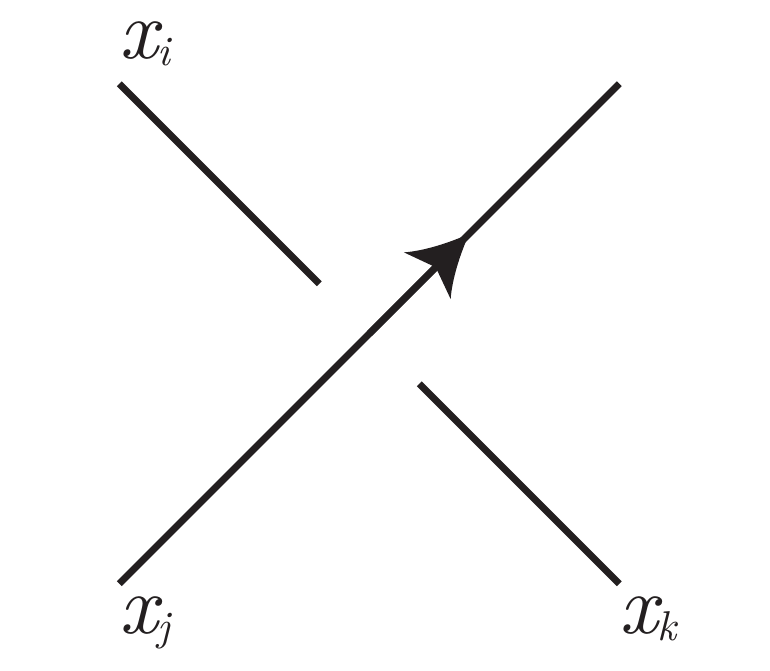}$$
\caption{The relation $x_i=x_k^{x_j}$ at a crossing.}
\label{crossing}
\end{figure}

The fundamental quandle of a link depends on the choice of orientation; however, its involutory quandle does not. Given an oriented link $L$ and a presentation $\langle S \,|\, R\rangle$ of $Q(L)$, a presentation of  the involutory quandle $Q_2(L)$ is obtained by adding the relations $x^{y^2}=x$ for every pair of distinct generators $x$ and $y$. Now the relation $x_i=x_k^{x_j}$ is equivalent to the relation $x_k=x_i^{x_j}$, so the involutory quandle does not depend on the orientation of the link.

Given a presentation of an involutory quandle, one can try to systematically enumerate its elements and simultaneously produce a Cayley graph of the quandle. Such a method was described in a graph-theoretic fashion by Winker in \cite{WI}. We provide a brief description of Winker's method applied to the involutory quandle of a link since it will be used extensively in this paper.  Suppose $Q_2(L)$ is presented as
$$Q_2(L)=\langle x_1, x_2, \dots, x_g \, |\, x_{j_1}^{w_1}=x_{k_1}, \dots, x_{j_r}^{w_r}=x_{k_r} \rangle_2,$$
where each $w_i$ is a word in the free group on $\{x_1,\dots , x_g\}$. Throughout this paper presentations of involutory quandles will not explicitly list the relations $x_i^{x_j^2}=x_i$ (nor the relations given by the quandle axioms) although we are implicitly assuming they hold. To avoid confusion we append the subscript $2$ to presentations of involutory quandles.

If $y$ is any element of the quandle, then it follows from the relation $x_{j_i}^{w_i}=x_{k_i}$ and Lemma~\ref{leftassoc} that $y^{\overline{w}_i x_{j_i}w_i}=y^{x_{k_i}}$, and so
$$y^{\overline{w}_i x_{j_i} w_i \overline{x}_{k_i}}=y.$$

Winker calls this relation the {\it secondary relation}  associated to the {\it primary relation} $x_{j_i}^{w_i}=x_{k_i}$. He also considers relations of the form $y^{ x_j^2 }=y$ for all $y$ and $1 \le j \le g$. These relations are equivalent to the secondary relations of the involutory quandle relations. In order to see this, notice that the secondary relation of the involutory quandle relation $x_k^{x_j^2}=x_k$ is $y=y^{\bar x_j^2 x_k x_j^2 \bar x_k}$ for all elements $y$.  Now given any $z$, if we let $y=z^{x_j^2}$ in this secondary relation, then $z^{x_j^2}=z^{x_j^2 \bar x_j^2 x_k x_j^2 \bar x_k} =  z^{x_k x_j^2 \bar x_k}$. Hence, $z^{x_j^2 x_k}=z^{x_k x_j^2}$ for all elements $z$. Now given any $y$ we have $y=x_i^w$ for some $1 \le i \le g$ and, by what we just observed, we can commute $x_j^2$ with $w$ in the exponent of $x_i$. Therefore,
$$y^{x_j^2} = x_i^{w x_j^2} = x_i^{x_j^2 w} = x_i^w =y.$$
Conversely, if $y^{x_j^2} =y$ for all $y$, then clearly $x_i^{x_j^2} =x_i$ as well.  

Winker's method now proceeds to build the Cayley graph associated to the presentation as follows:
  
\begin{enumerate}
\item Begin with $g$ vertices labeled $x_1,x_2, \dots, x_g$ and numbered $1,2, \dots,g$. 
\item Add an loop at each vertex $x_i$ and label it $x_i$. (This encodes the axiom A1.)
\item For each value of $i$ from $1$ to $r$, {\em trace} the primary relation $x_{j_i}^{w_i}=x_{k_i}$ by introducing new vertices and edges as necessary to create a path from $x_{j_i}$ to $x_{k_i}$ given by $w_i$. Consecutively number (starting from $g+1$) new vertices in the order they are introduced.  Edges are labelled with their corresponding generator. (Note that, for an involutory quandle, the edges are unoriented.)
\item Tracing a relation may introduce edges with the same label at a shared vertex. We identify all such edges, possibly leading to other identifications. This process is called {\it collapsing} and all collapsing is carried out before tracing the next relation. 
\item Proceeding in order through the vertices, trace and collapse each $n$-quandle relation $y^{x_j^2}=y$  and each secondary relation (in order). All of these relations are traced and collapsed at a vertex before proceeding to the next vertex.
\end{enumerate}

The method will terminate in a finite graph if and only if the involutory quandle is finite. The reader is referred to Winker~\cite{WI} and Hoste and Shanahan \cite{HS3} for more details.

%Associated to every quandle $Q$ is its automorphism group $\text{Aut} (Q)$. The {\it inner} automorphism group of $Q$, denoted by $\text{Inn}(Q)$,  is the  normal subgroup of $\text{Aut}(Q)$ generated by the point symmetries $S_x$.  The {\it transvection} group of $Q$, denoted by $\text{Trans}(Q)$, is the subgroup of $\text{Inn}(Q)$ generated by all products $S_x S_y^{-1}$. The subgroup $\text{Trans}(Q)$ is normal in both $\text{Aut} (Q)$ and $\text{Inn}(Q)$. Moreover, $\text{Trans}(Q)$ is abelian if and only if 
%\begin{equation}\label{abelian}(x^y)^{(z^w)}=(x^z)^{(y^w)}\end{equation}
%for all elements $x, y, z, w\in Q$. Quandles that satisfy the property \eqref{abelian} are called {\it medial} or {\it abelian}. See \cite{JO} for more details. 

%Some results in this paper were obtained using the RIG package for GAP.  The Cayley graph of a finite quandle $Q$ can be used to produce the operation table for $\rhd$ which is encoded in a matrix $M_Q$. In RIG, a rack (or quandle) can then be defined using the command {\sc Rack}$(M_Q)$. Once the quandle is entered into RIG, the built-in commands {\sc AutomorphismGroup}, {\sc InnerGroup}, and {\sc TransvectionsGroup} will compute $\text{Aut} (Q)$, $\text{Inn} (Q)$, and $\text{Trans} (Q)$, respectively.  Finally, the GAP command {\sc StructureDescription} will determine the structure of the group, such as $\mathbb Z_2 \times S_4$.  No additional special code is required to reproduce our results.

\section{The links $L(k, p/q) \cup C$} \label{S:link}

The link $L(k, p/q) \cup C$ consists of the two-bridge link $L(k, p/q)$ linked with an unknotted component $C$, as in Figure \ref{F:L(k,p,q)UC}.  Here $k$ represents a number of right-handed half-twists, and $p/q$ represents a $p/q$-tangle. (It is not difficult to show that the link $L(k,p/q)$ is the two-bridge link with fraction $q/(p-kq)$, but this is not needed for our analysis.) We assume $q > 0$ and $\gcd(p,q) = 1$ (otherwise there will be additional components inside the $p/q$-tangle).  For example, Figure \ref{F:L(3,3/5)UC} shows the link $L(3, 3/5) \cup C$.  Using flypes, we can change $p$ by a multiple of $q$ at the expense of changing $k$.  In particular, if we replace $p$ by $p \pm q$, then we replace $k$ by $k \pm 1$.  So we may assume $0 < p < q$ (the case of $p = 0$, when the tangle is trivial, was considered in \cite{CHMS}).

\begin{figure}[htbp]
\centerline{\scalebox{.75}{\includegraphics{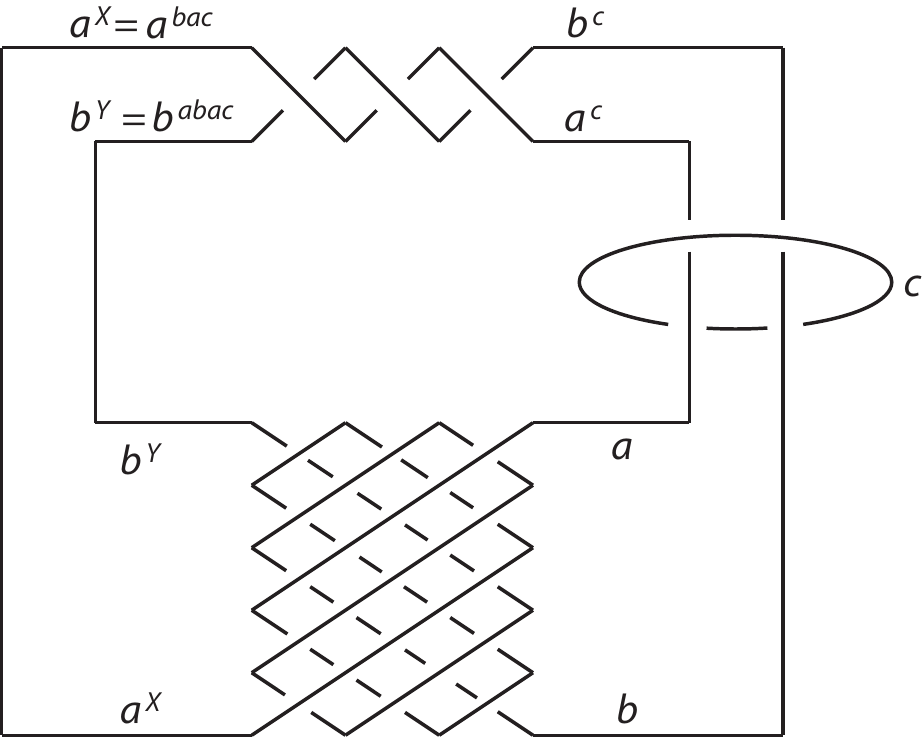}}}
\caption{The link $L(3, 3/5) \cup C$.}
\label{F:L(3,3/5)UC}
\end{figure}

We begin by finding a presentation for the involutory quandle $Q_2(L)$ for the link $L = L(k, p/q) \cup C$.  We label three arcs with quandle elements $a, b$ and $c$, as shown in Figure \ref{F:L(3,3/5)UC}; these will be the generators of our presentation. There are three relations in $Q_2(L)$.  The first, $R1$, comes from following the component $C$ as it crosses under $L(k, p/q)$:

$$R1: \qquad c^{ab} = c \ {\rm or}\ c^a = c^b$$

This relation has some useful corollaries.

\begin{lemma}\label{L:r1}
For any $z \in Q_2(L)$, we have
\begin{enumerate}
	\item $z^{aca} = z^{bcb}$
	\item $z^{cab} = z^{abc}$ and $z^{cba} = z^{bac}$
	\item For all $w \in \{ca, ac, cb, bc\}$, $z^{wab} = z^{baw}$ and $z^{wba} = z^{abw}$.
	\item For all $w \in \{ca, ac, cb, bc\}$, $z^{w^2ab} = z^{abw^2}$ and $z^{w^2ba} = z^{baw^2}$.
\end{enumerate}
\end{lemma}
\begin{proof}
Since $c^a = c^b$, we immediately get $z^{aca} = z^{c^a} = z^{c^b} = z^{bcb}$.  Then $z^{cab} = z^{a(aca)b} = z^{a(bcb)b} = z^{abc}$.  Similarly, $z^{cba} = z^{b(bcb)a} = z^{b(aca)a} = z^{bac}$.

For part (3), we consider the case when $w = ca$.  Then, using part (1), $z^{caab} = z^{cb} = z^{bbcb} = z^{baca}$ and $z^{caba} = z^{aacaba} = z^{abcbba} = z^{abca}$.  The other cases are proved similarly.  Part (4) is just the application of (3) twice.
\end{proof}

The remaining relations come from following the strands of the two-bridge knot through the block of half-twists, and then through the $p/q$-tangle.  The arcs on the right side of the block of half-twists are labeled $a^c$ and $b^c$, as in Figure \ref{F:L(3,3/5)UC}. The labels on the arcs on the left side of the block are easily determined by an inductive argument.

\begin{lemma} \label{L:halftwists}
The arcs on either side of the block of $k$ right-handed half-twists are labeled as shown below (for $k$ even and $k$ odd).  Here $X = (ba)^tc$ and $Y = (ba)X = (ba)^{t+1}c$. (If $k < 0$, there are $\vert k \vert$ left-handed half-twists; the same formulas hold, where $(ba)^{-1} = \overline{ba} = ab$.)
$$\scalebox{1}{\includegraphics{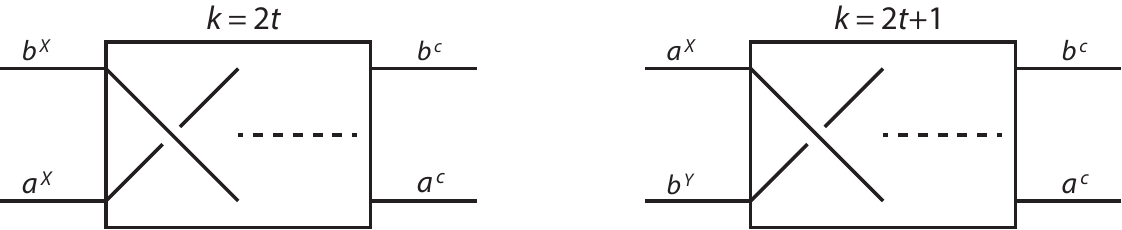}}$$
\end{lemma}

The following lemma will be useful (as in Lemma \ref{L:halftwists}, $X = (ba)^tc$ and $Y = (ba)X = (ba)^{t+1}c$):

\begin{lemma} \label{L:bYaX}
For any $z \in Q_2(L)$, $z^{b^Y a a^X} = z^{a^X a b^Y}$ and $z^{b^X a a^X} = z^{a^X a b^X}$.
\end{lemma}
\begin{proof}
We will do the proof for the first case; the second case.  We use Lemmas \ref{leftassoc} and \ref{L:r1}:
$$z^{b^Y a a^X} = z^{\bar Y b Y a \bar X a X} = z^{c(ab)^{t+1}b(ba)^{t+1}c a c(ab)^t a (ba)^t c} = z^{c(ab)^{2t+1}ac a c(ab)^{2t} a c}= z^{c(ab)^{4t+1}(ac)^3}$$
$$z^{a^X a b^Y} = z^{\bar X a X a \bar Y b Y} = z^{c(ab)^t a (ba)^t c a c(ab)^{t+1}b(ba)^{t+1}c} = z^{c(ab)^{2t}acac(ab)^{2t+1}ac} = z^{c(ab)^{4t+1}(ac)^3}$$
Hence $z^{b^Y a a^X} = z^{a^X a b^Y}$.

Similarly, $z^{b^X a a^X} = z^{a^X a b^X}$.
\end{proof}

To see the effect of the $p/q$-tangle, we redraw the tangle as shown in Figure \ref{F:tangle}.  Notice we now have $q-1$ arcs crossing under the arc labeled $a$, $p$ arcs crossing under the arc $a^X$, and $(q-1)-p$ arcs crossing under the arc $b^Y$.  The strands alternate between crossing under $a$ and crossing under $a^X$ or $b^Y$.  For example, in the $3/5$-tangle, we get the relations $a^X = a^{a^Xab^Ya}$ and $b^Y = b^{a^Xa a^Xa}$.  By Lemma \ref{L:bYaX} we can rearrange the factors $a^Xa$ and $b^Ya$ in the exponents as desired.

\begin{figure}[htbp]
\centerline{\scalebox{.75}{\includegraphics{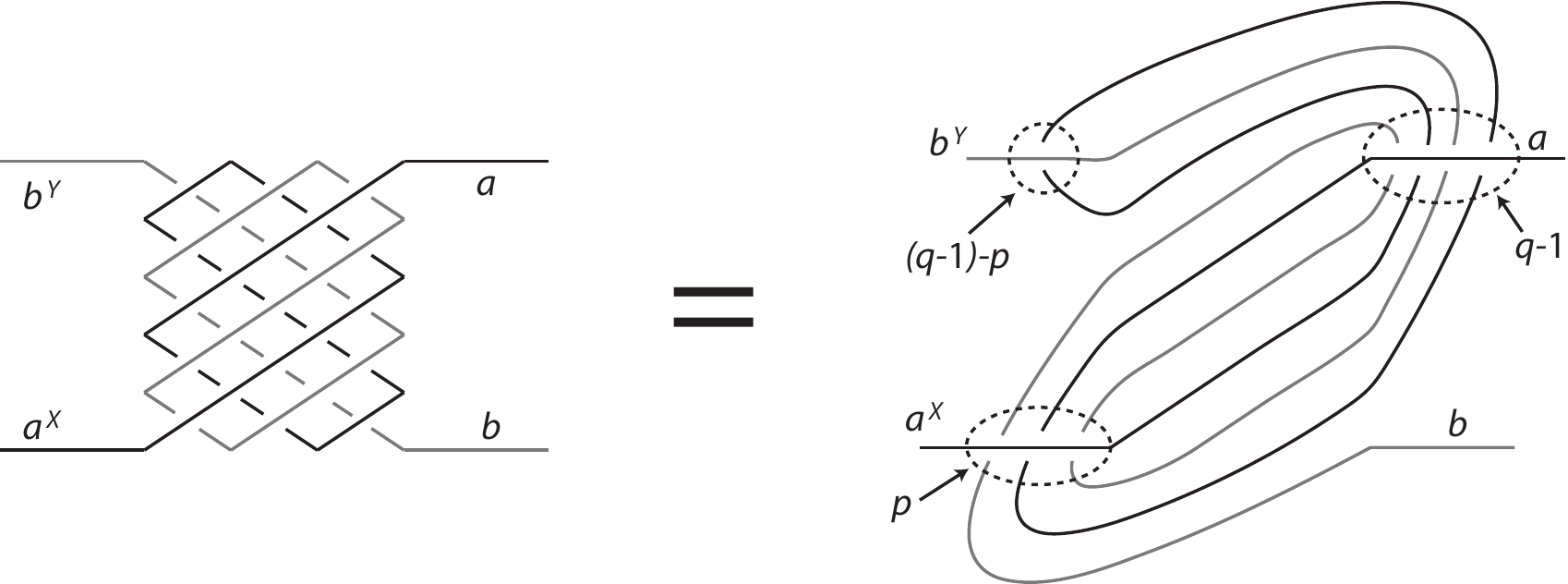}}}
\caption{Redrawing the $p/q$ tangle (illustrated with the 3/5-tangle).}
\label{F:tangle}
\end{figure}

To see how many of each factor there are, we can imagine drawing the path from $a$ to $a^X$ (or $b$ or $b^Y$) as a straight line of slope 1 on a plane tiled with $p$ by $q$ rectangles, as in Figure \ref{F:grid}.  The vertical lines alternately represent the left and right sides of the $p/q$ tangle, and the horizontal lines alternately represent the top and bottom of the tangle.  Every time the line crosses the right side (except for the starting point), there is a factor of $a$; every time it crosses the bottom, there is a factor of $a^X$ (for $k$ odd; $b^X$ if $k$ is even).  This allows us to write the relations induced by the $p/q$-tangle, based on the parities of $k, p, q$.  As in Lemma \ref{L:halftwists}, $X = (ba)^tc$ and $Y = (ba)^{t+1}c$, where $t = k/2$ if $k$ is even and $t = (k-1)/2$ if $k$ is odd. (Since $\gcd(p,q) = 1$, $p$ and $q$ cannot both be even.)

\begin{figure}[htbp]
\centerline{\scalebox{.75}{\includegraphics{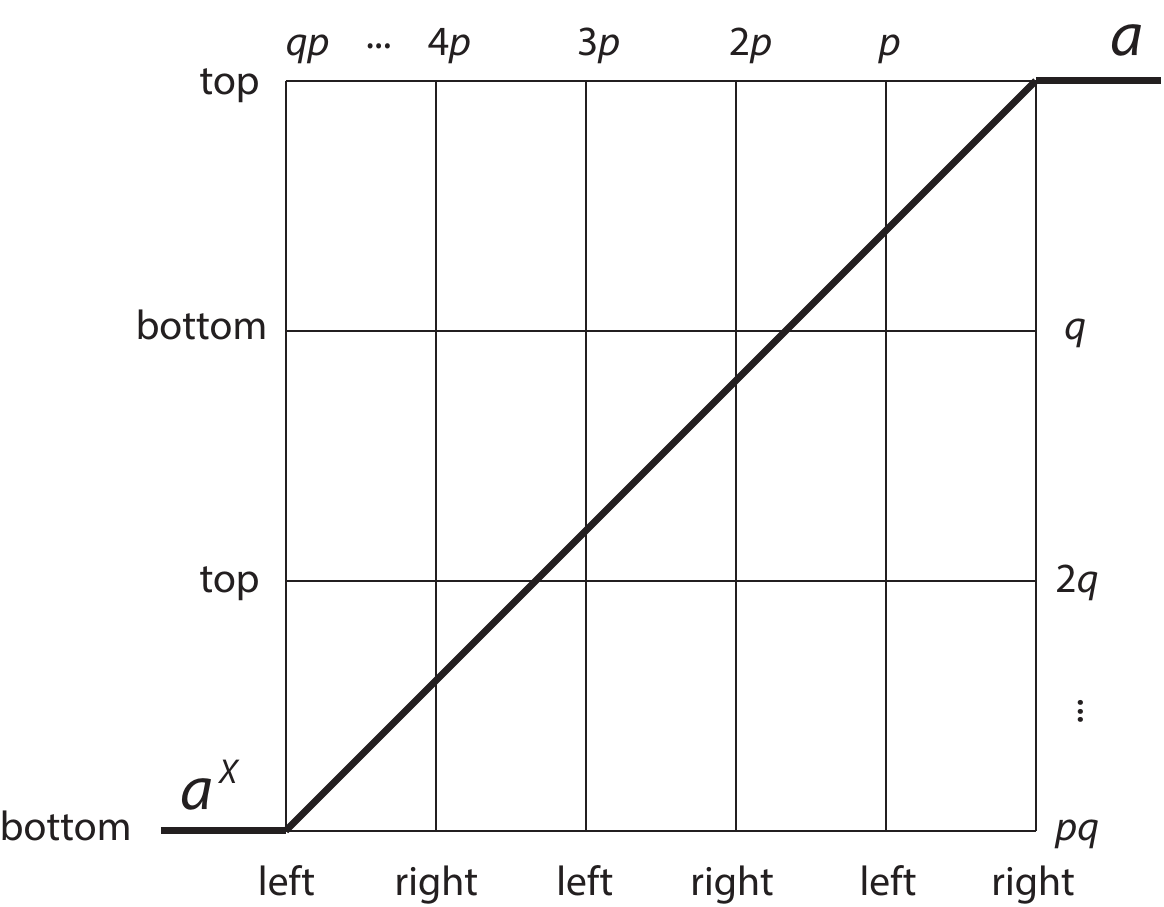}}}
\caption{Drawing the path from $a$ to $a^X$ on a grid (illustrated for the 3/5-tangle).}
\label{F:grid}
\end{figure}

\begin{center}
\begin{tabular}{c|c|c|c|c}
$k$ & $p$ & $q$ & $R2$ & $R3$ \\ \hline
\rule[-.3\baselineskip]{0pt}{.25in} odd & odd & odd & $a^{(a^X a)^{(p-1)/2} (b^Y a)^{(q-p)/2}} = a^X$ & $b^{(a^X a)^{(p+1)/2} (b^Y a)^{(q-p-2)/2}} = b^Y$ \\
\rule[-.3\baselineskip]{0pt}{.25in} odd & even & odd & $a^{(a^X a)^{p/2} (b^Y a)^{(q-p-1)/2}} = b^Y$ & $b^{(a^X a)^{p/2} (b^Y a)^{(q-p-1)/2}} = a^X$ \\
\rule[-.3\baselineskip]{0pt}{.25in} odd & odd & even & $a^{(a^X a)^{(p-1)/2} (b^Y a)^{(q-p-1)/2}a^X} = b$ & $(a^X)^{(aa^X)^{(p-1)/2} (a b^Y)^{(q-p-1)/2}a} = b^Y$ \\
\rule[-.3\baselineskip]{0pt}{.25in} even & odd & odd & $a^{(b^X a)^{(p-1)/2} (a^X a)^{(q-p)/2}} = b^X$ & $b^{(b^X a)^{(p+1)/2} (a^X a)^{(q-p-2)/2}} = a^X$ \\
\rule[-.3\baselineskip]{0pt}{.25in} even & even & odd & $a^{(b^X a)^{p/2} (a^X a)^{(q-p-1)/2}} = a^X$ & $b^{(b^X a)^{p/2} (a^X a)^{(q-p-1)/2}} = b^X$ \\
\rule[-.3\baselineskip]{0pt}{.25in} even & odd & even & $a^{(b^X a)^{(p-1)/2} (a^X a)^{(q-p-1)/2}b^X} = b$ & $(b^X)^{(a b^X)^{(p-1)/2} (a a^X)^{(q-p-1)/2}a} = a^X$ \\
\end{tabular}
\end{center}

Now we replace $X$ with $(ba)^t c$ and $Y$ with $(ba)^{t+1} c$ to derive the relations only in terms of $a, b, c$.  Using Lemmas \ref{leftassoc} and \ref{L:r1}, we do this for relation $R2$ when $k, p, q$ are all odd:
\begin{align*}
a^{(ba)^t c} &= a^{(c(ab)^t a (ba)^t c a)^{(p-1)/2} (c (ab)^{t+1} b (ba)^{t+1} c a)^{(q-p)/2}} \\
&= a^{(c(ab)^{2t}aca)^{(p-1)/2}(c(ab)^{2t+1}aca)^{(q-p)/2}} \\
&= a^{((ab)^{2t}(ca)^2)^{(p-1)/2}((ab)^{2t+1}(ca)^2)^{(q-p)/2}} \\
&= a^{(ab)^{(kq-p)/2-t}(ca)^{q-1}}\\
\implies a &= a^{(ab)^{(kq-p)/2-t}(ca)^{q-1}c(ab)^t} = a^{(ab)^{(kq-p)/2}(ca)^{q-1}c}\\
\implies a^{c(ac)^{q-1}} &= a^{(ab)^{(kq-p)/2}} = a^{(ba)^{(kq-p)/2}a} \\
\implies a^{(ca)^q} &= a^{(ba)^{(kq-p)/2}}
\end{align*}

Note that this works regardless of whether $kq-p$ is positive or negative (in particular, the reader can check that for any $x, y, z$, $z^{(xy)^m} = z^{x(yx)^mx}$, regardless of whether $m$ is positive or negative).  The derivations in the other cases are similar.  The relations $R2$ and $R3$ are as follows; we find that they depend only on the parity of $kq-p$:

\begin{center}
\begin{tabular}{c|c|c}
$kq-p$  & $R2$ & $R3$ \\ \hline
\rule[-.3\baselineskip]{0pt}{.25in} odd & $a^{(ca)^{q}} = b^{(ab)^{(kq-p-1)/2}}$ & $a^{(ac)^{q}} = b^{(ab)^{(kq-p-1)/2}}$ \\
\rule[-.3\baselineskip]{0pt}{.25in} even & $a^{(ca)^q} = a^{(ba)^{(kq-p)/2}}$ & $b^{(bc)^{q}} = b^{(ab)^{(kq-p)/2}}$ 
\end{tabular}
\end{center}

\section{The Cayley graph of $Q_2(L)$}

Now we will construct the Cayley graph for the quandle $Q_2(L)$ presented in the last section. There are two cases, depending on whether $kq-p$ is odd or even. The Cayley graphs are shown in Figures \ref{F:odd} and \ref{F:even}.

\begin{figure}[htbp]
\centerline{\scalebox{.75}{\includegraphics{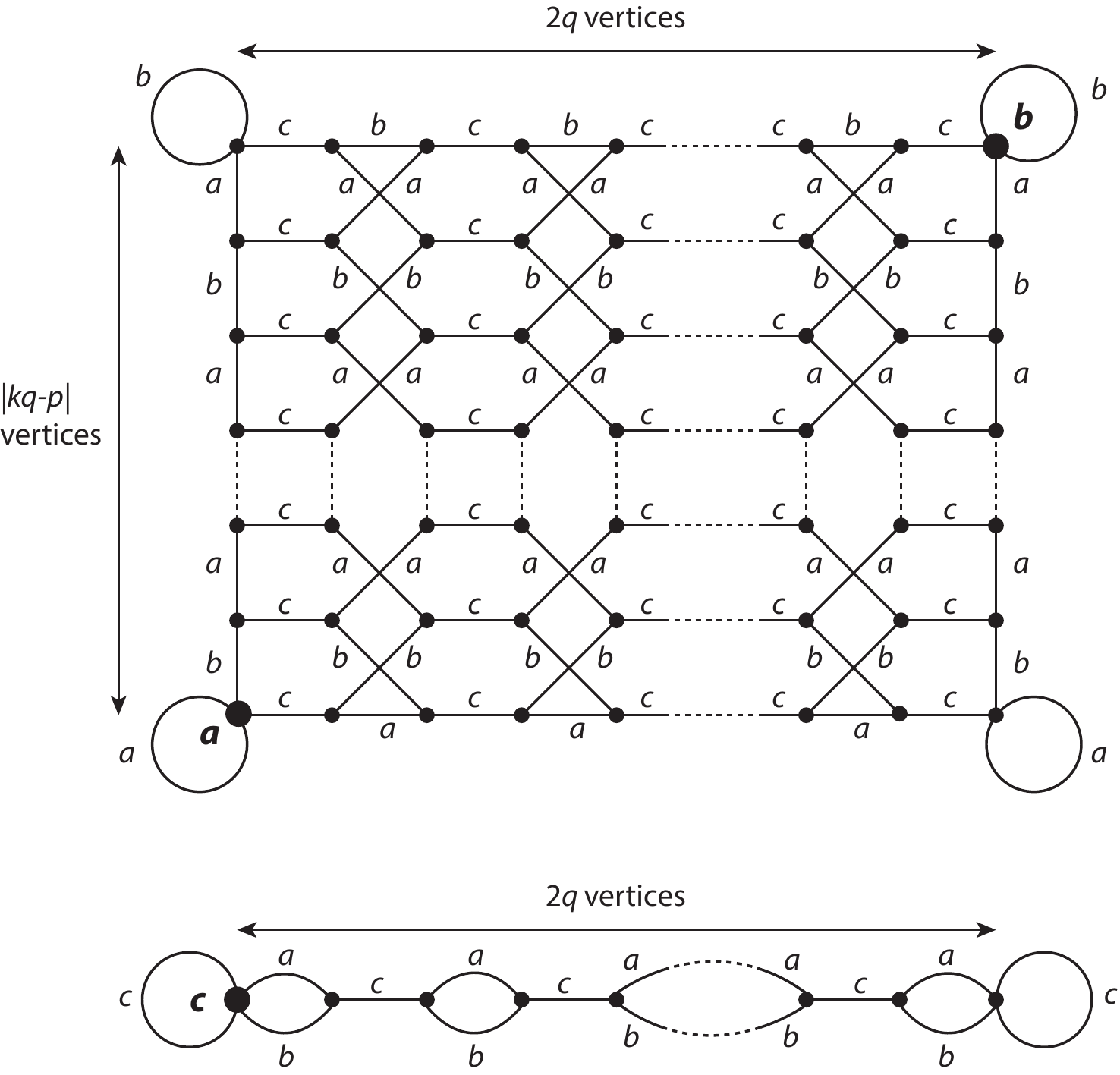}}}
\caption{The Cayley graph for $Q_2(L)$ with $kq-p$ odd. The generators $a$, $b$ and $c$ are indicated with larger vertices.}
\label{F:odd}
\end{figure}

\begin{figure}[htbp]
\centerline{\scalebox{.75}{\includegraphics{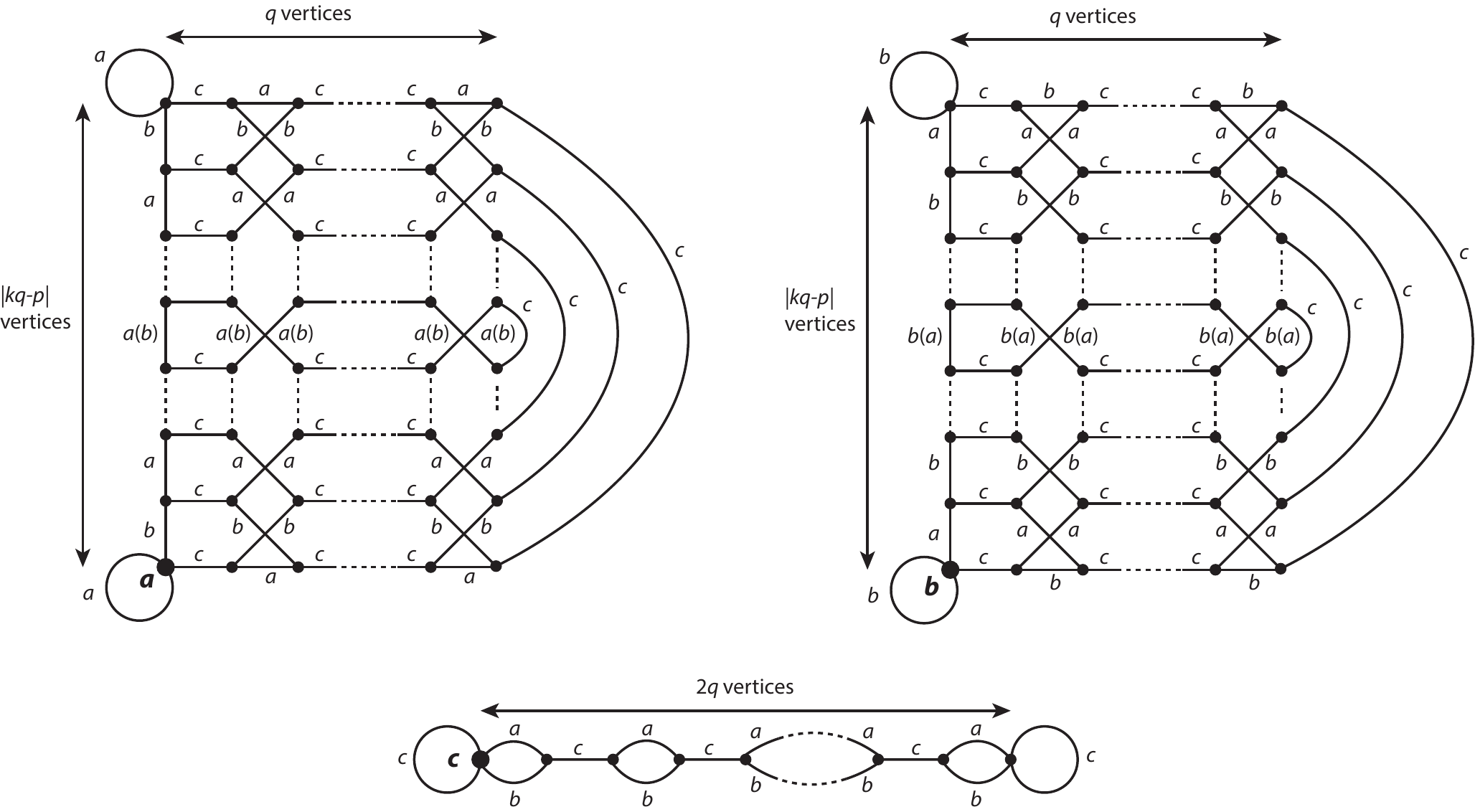}}}
\caption{The Cayley graph for $Q_2(L)$ with $kq-p$ even. The generators $a$, $b$ and $c$ are indicated with larger vertices.}
\label{F:even}
\end{figure}

\begin{theorem} \label{T:Cayley}
Suppose $L = L(k, p/q) \cup C$.  Then the Cayley graph of $Q_2(L)$ is the 2-component graph in Figure \ref{F:odd} when $kq-p$ is odd, and is the 3-component graph in Figure \ref{F:even} when $kq-p$ is even.
\end{theorem}

\begin{corollary} \label{C:cardinality}
$\vert Q_2(L)\vert = 2q(\vert kq-p\vert +1).$
\end{corollary}

To prove Theorem \ref{T:Cayley}, we will find additional relations in $Q_2(L)$ before applying Winker's algorithm. 

\subsection{$kq-p$ odd}

We will first consider the case when $kq-p$ is odd. 

\begin{lemma} \label{L:odd}
When $kq-p$ is odd, the following relations hold in $Q_2(L)$ (in addition to $R1$, $R2$ and $R3$).
\begin{align*}
R4 &: c^{(ac)^{2q}} = c \\
R5 &: a^{(ca)^{2q}} = a \\
R6 &: b^{(cb)^{2q}} = b \\
R7 &: a^{(ba)^{(kq-p-1)/2}} = b^{(bc)^{q}} \\
\alpha_{i,j} &: a^{(ca)^i(ba)^j c} = a^{(ca)^i c (ba)^j} (0 \leq i \leq q; 0 \leq j \leq (\vert kq-p\vert -1)/2) \\
\beta_{i,j} &: a^{(ca)^i(ab)^j abc} = a^{(ca)^{i+1} (ba)^j b} (0 \leq i \leq q; 0 \leq j \leq (\vert kq-p\vert -1)/2) \\
\gamma_i &: c^{(ac)^i a} = c^{(ac)^i b} (0 \leq i) 
\end{align*}
\end{lemma}

\begin{proof}
Let's recall the defining relations for $Q_2(L)$ (in addition to the relations common to any involutory quandle):
\begin{align*}
R1 &: c^{ab} = c \\
R2 &: a^{(ca)^{q}} = b^{(ab)^{(kq-p-1)/2}} \\
R3 &: a^{(ac)^{q}} = b^{(ab)^{(kq-p-1)/2}} 
\end{align*}
We will first prove relation $\gamma_i$. Observe that $c^{(ac)^i ab} = c^{ab(ac)^i}$ if $i$ is even, or $c^{ba(ac)^i}$ if $i$ is odd, by Lemma \ref{L:r1}.  In either case, by $R1$, $c^{(ac)^i ab} = c^{(ac)^i}$, so $c^{(ac)^i a} = c^{(ac)^i b}$, proving $\gamma_i$.

The secondary relation associated to relation $R3$ is (for any $x \in Q_2(L)$):
$$x^{a^{(ac)^{q}}} = x^{b^{(ab)^{(kq-p-1)/2}}} \implies x^{(ca)^{q}a(ac)^{q}} = x^{(ba)^{(kq-p-1)/2}b (ab)^{(kq-p-1)/2}}$$
$$\implies x^{c(ac)^{2q-1}} = x^{(ba)^{kq-p-1}b}$$
In particular, if $x = c$, we get
$$c^{c(ac)^{2q-1}} = c^{(ba)^{kq-p-1}b} \implies c^{(ac)^{2q-1}} = c^b\ ({\rm since\ }c^{ba} = c)$$
$$\implies c^{(ac)^{2q}} = c^{bac} = c^c = c$$
This proves relation $R4$.

Combining relations $R2$ and $R3$ gives $a^{(ca)^q} = a^{(ac)^{q}}$.  This implies $a^{(ca)^{2q}} = a$, proving relation $R5$.  

If $q$ is {\em odd}, we can rewrite $R2$ as follows (using Lemma \ref{L:r1}):
\begin{align*}
a^{(ca)^{q}} = b^{(ab)^{(kq-p-1)/2}} &\implies a^{(ca)^{q}(ba)^{(kq-p-1)/2}} = b \\
&\implies a^{(ab)^{(kq-p-1)/2}(ca)^{q}} = b \\ 
&\implies a^{(ab)^{(kq-p-1)/2}} = b^{(ac)^{q}}
\end{align*}
Similarly, we rewrite $R3$:
\begin{align*}
a^{(ac)^{q}} = b^{(ab)^{(kq-p-1)/2}} &\implies a^{(ac)^{q}(ba)^{(kq-p-1)/2}} = b \\
&\implies a^{(ab)^{(kq-p-1)/2}(ac)^{q}} = b \\
&\implies a^{(ab)^{(kq-p-1)/2}} = b^{(ca)^{q}}
\end{align*}
In the same way, if $q$ is {\em even}, we find $b^{(ac)^{q}} = b^{(ca)^{q}} = a^{(ba)^{(kq-p-1)/2}}$.  In either case, we get $b^{(ca)^{2q}} = b$.  But for any $x \in Q_2(L)$, we have $x^{(ca)^2} = x^{caca} = (x^c)^{aca} = (x^c)^{bcb} = x^{(cb)^2}$ (by Lemma \ref{L:r1}).  Hence we get $b^{(cb)^{2q}} = b$, proving $R6$.

The derivation of $R7$ also depends slightly on the parity of $q$. If $q$ is {\em odd} (so $q-1$ is even), then 
$$a^{(ba)^{(kq-p-1)/2}} = a^{(ab)^{(kq-p-1)/2}a} = b^{(ca)^{q}a} = b^{b(ca)^{q-1}c} = b^{b(cb)^{q-1}c} = b^{(bc)^{q}}.$$
If $q$ is even, then
$$a^{(ba)^{(kq-p-1)/2}} = b^{(ac)^{q}} = b^{(bc)^{q}}.$$

To derive $\alpha_{i,j}$, observe that (using Lemma \ref{L:r1})
$$a^{(ca)^i(ba)^j c (ab)^j} = a^{(ca)^i c^{(ab)^j}} = a^{(ca)^i c} \implies a^{(ca)^i(ba)^j c} = a^{(ca)^i c (ba)^j}.$$

Finally, to derive $\beta_{i,j}$, we observe that
$$a^{(ca)^i(ab)^j abc b (ab)^j a} = a^{(ca)^i c^{(ba)^{j+1}}} = a^{(ca)^i c} = a^{(ca)^{i+1} a} \implies a^{(ca)^i(ab)^j abc} = a^{(ca)^{i+1} (ba)^j b}.$$
\end{proof}

Now we can prove the first part of Theorem \ref{T:Cayley}. We first consider the larger component of the Cayley graph, containing generators $a$ and $b$.  Applying Winker's algorithm, we see that relation $R3$ traces out the bottom and right side of the grid, and relation $R7$ traces out the left side and the top.  (Note that if $kq-p < 0$, then $\vert kq-p-1\vert = \vert kq-p\vert + 1$.) Tracing and collapsing relations $R5$ and $R6$ gives the loops in the lower right and upper left corners.  Finally, tracing and collapsing $\alpha_{i,j}$ and $\beta_{i,j}$ gives the interior of the grid.  Note that $R2$ also traces out the bottom and right side of the grid (including the loop in the lower right corner).

For the second component, containing generator $c$, relation $\gamma_i$ gives a chain of bigons connected by edges labeled $c$, and relation $R4$ restricts the length of the chain and gives the loop on the right side.  Note that $R1 = \gamma_0$, so it is also satisfied.

It only remains to show that the secondary relations do not cause additional collapsing.  We denote the secondary relations for $R1$, $R2$ and $R3$ by $SR1$, $SR2$ and $SR3$ respectively. For any $x \in Q_2(L)$, these relations are:
\begin{align*}
SR1 &: x^{bacabc} = x \\
SR2 &: x^{(ac)^{2q}} = x^{(ba)^{kq-p}} \\
SR3 &: x^{(ca)^{2q}} = x^{(ba)^{kq-p}}
\end{align*}

The relation $SR1$ is satisfied by the hexagons in the grid.  Relations $SR2$ and $SR3$ are satisfied since $x^{(ac)^{2q}} = x^{(ca)^{2q}} = x^{(ba)^{kq-p}} = x$ at every point in the grid (making a complete circuit of the grid either horizontally or vertically).  It is easy to see that the relations are satisfied in the smaller component as well.  Hence the graph shown in Figure \ref{F:odd} is the Cayley graph of $Q_2(L)$ when $kq-p$ is odd.

\subsection{$kq-p$ even}

We now turn to the case when $kq-p$ is even.

\begin{lemma} \label{L:even}
When $kq-p$ is even, the following relations hold in $Q_2(L)$ (in addition to $R1$, $R2$ and $R3$).
\begin{align*}
R4 &: c^{(ac)^{2q}} = c \\
R5 &: a^{(ca)^{2q}} = a \\
R6 &: b^{(cb)^{2q}} = b \\
R7 &: a^{(ba)^{(kq-p)/2}} = a^{(ac)^{q}} \\
R8 &: b^{(ab)^{(kq-p)/2}} = b^{(bc)^{q}} \\
\alpha_{i,j,a} &: a^{(ca)^i(ba)^j c} = a^{(ca)^i c (ba)^j} (0 \leq i \leq q; 0 \leq j \leq \vert kq-p\vert/2) \\
\alpha_{i,j,b} &: b^{(cb)^i(ab)^j c} = b^{(cb)^i c (ab)^j} (0 \leq i \leq q; 0 \leq j \leq \vert kq-p\vert/2) \\
\beta_{i,j,a} &: a^{(ca)^i(ab)^j abc} = a^{(ca)^{i+1} (ba)^j b} (0 \leq i \leq q; 0 \leq j \leq \vert kq-p\vert/2) \\
\beta_{i,j,b} &: b^{(cb)^i(ba)^j bac} = b^{(cb)^{i+1} (ab)^j a} (0 \leq i \leq q; 0 \leq j \leq \vert kq-p\vert/2) \\
\gamma_i &: c^{(ac)^i a} = c^{(ac)^i b} (0 \leq i) 
\end{align*}
\end{lemma}
\begin{proof}
Recall that when $kq-p$ is even, the defining relations for $Q_2(L)$ are:
\begin{align*}
R1 &: c^{ab} = c \\
R2 &: a^{(ca)^{q}} = a^{(ba)^{(kq-p)/2}} \\
R3 &: b^{(bc)^{q}} = b^{(ab)^{(kq-p)/2}} 
\end{align*}
The secondary relations associated to these, for any $x \in Q_2(L)$, are:
\begin{align*}
SR1 &: x^{bacabc} = x \\
SR2 &: x^{(ac)^{2q}} = x^{(ab)^{kq-p}} \\
SR3 &: x^{(cb)^{2q}} = x^{(ba)^{kq-p}}
\end{align*}
Since $p$ and $q$ cannot both be even, $kq-p$ is even only when $k, q$ and $p$ are all odd, or when $k$ and $p$ are even and $q$ is odd.  So, in either case, we may assume $q$ is odd.

We first note that $R1$ implies the relations $\gamma_i$, as in Lemma \ref{L:odd}, and that $R1$ and $SR2$ imply $R4$ (letting $x = c$ in $SR2$).

Replacing $x$ with $a$ in $SR2$ gives
$$a^{(ac)^{2q}} = a^{(ab)^{kq-p}} \implies a^{a(ca)^{2q}a} = a^{a(ba)^{kq-p}a} \implies a^{(ca)^{2q}} = a^{(ba)^{kq-p}}$$
Now we can apply relation $R2$ to derive $R5$:
$$a^{(ca)^{2q}} = a^{(ba)^{kq-p}} = a^{(ba)^{(kq-p)/2}(ba)^{(kq-p)/2}} \underset{R2}{=} a^{(ca)^q(ba)^{(kq-p)/2}}$$
$$\underset{q\ odd}{=} a^{(ab)^{(kq-p)/2}(ca)^q} = a^{(ba)^{(kq-p)/2}(ac)^q a} \underset{R2}{=} a^{(ca)^q(ac)^qa} = a.$$
Similarly, we combine $SR3$ and $R3$ to derive $R6$.

$R7$ comes directly from $R2$, using the fact that $q$ is odd:
$$a^{(ca)^{q}} = a^{(ba)^{(kq-p)/2}} \implies a = a^{(ba)^{(kq-p)/2}(ac)^{q}} \underset{q\ odd}{=} a^{(ac)^{q}(ab)^{(kq-p)/2}} \implies a^{(ba)^{(kq-p)/2}} = a^{(ac)^{q}}.$$
Similarly, $R8$ is derived from $R3$.

The relations $\alpha_{i,j,a}$ and $\alpha_{i,j,b}$ follow immediately from part (2) of Lemma \ref{L:r1}.  Relations $\beta_{i,j,a}$ (and, similarly, $\beta_{i,j,b}$) also follow from Lemma \ref{L:r1}:
$$a^{(ca)^i(ab)^j abc} = a^{(ca)^i(ab)^j ab(ca)a} = a^{(ca)^{i+1}(ba)^j baa} = a^{(ca)^{i+1} (ba)^j b}.$$
\end{proof}

We complete the proof of Theorem \ref{T:Cayley} by showing the Cayley graph of $Q_2(L)$ when $kq-p$ is even is shown in Figure \ref{F:even}.  For the component containing generator $c$, as when $kq-p$ is odd, relation $\gamma_i$ gives a chain of bigons connected by edges labeled $c$, and relation $R4$ restricts the length of the chain and gives the loop on the right side.  Relation $R1$ is also satisfied.

For the component containing generator $a$, relation $R7$ traces the path around the outside of the graph, and relation $R5$ gives the loop in the top left (relation $R2$ also traces around the outside of the graph, including the loop).  Tracing and collapsing the relations $\alpha_{i,j,a}$ and $\beta_{i,j,a}$ fills in the interior of the graph.  The relations $R8$, $R6$, $\alpha_{i,j,b}$ and $\beta_{i,j,b}$ do the same for the component containing $b$.

Finally, it is straightforward to check that, as in the case when $kq-p$ is odd, all of the secondary relations are satisfied with no further collapsing.  Hence the graph shown in Figure \ref{F:even} is the Cayley graph of $Q_2(L)$ when $kq-p$ is even.

\section{Classifying links using finite involutory quandles}

In Table \ref{orders}, we summarize the results on links with finite involutory quandles from \cite{CHMS}, \cite{HS1} and the current paper, listing the orders of the quandles, the number of components and the orders of each component.  In many cases, this information is sufficient to distinguish the links.  For example, Hoste and Shanahan \cite{HS1} showed that two $(2,2,r)$-Montesinos links with the same involutory quandle are either equivalent or mirror images.  We can prove a similar result for $L(k,p/q) \cup C$.

\begin{theorem} \label{classify}
If unoriented links $L = L(k,p/q) \cup C$ and $L' = L(k', p'/q') \cup C$ have isomorphic involutory quandles, then $L$ and $L'$ are either equivalent or mirror images.
\end{theorem}

\begin{proof}
Comparing the orders of the components of the involutory quandle, we immediately see that $q = q'$ and $\vert kq-p \vert = \vert k'q'-p' \vert$, so $\vert kq-p\vert = \vert k'q-p'\vert$.  There are two cases: $kq-p = k'q-p'$ and $kq-p = p' - k'q$.

As explained in Section \ref{S:link}, we may assume $0 < p, p' < q$.  Then, if $kq-p = k'q-p'$, the division algorithm implies $k = k'$ and $p = p'$, so $L$ and $L'$ are equivalent. 

On the other hand, if $kq-p = p'-k'q$, then $p'+p = (k+k')q$.  Since $0 < p'+p < 2q$, this means $p'+p = q$ and $k+k' = 1$, so $p' = q-p$ and $k' = 1-k$.  But then $L(k', p'/q) = L(1-k, (q-p)/q) = L((1-k)-1, ((q-p)-q)/q) = L(-k, -p/q)$, which is the mirror image of $L(k, p/q)$.  So $L'$ is the mirror image of $L$.
\end{proof}

\begin{table}
\begin{center}
\renewcommand{\arraystretch}{1.5}
\begin{tabular}{|c||c|c|c|}
\hline
Link & $|Q_2|$ & \# of Components & Orders of components  \\ \hline
$L_{p/q}$, $q$ odd & $q$ & 1 & $q$  \\ \hline
$L_{p/q}$, $q$ even & $q$ & 2 & $q/2$, $q/2$  \\ \hline
$T_{3,3}$ & 6 & 3 & 2, 2, 2  \\ \hline
$T_{3,4}$ & 12 & 1 & 12  \\ \hline
$T_{3,5}$ & 30 & 1 & 30  \\ \hline
$T_{2,q} \cup A$, $q$ odd & $2+2|q|$ & 2 & $2|q|$, 2 \\ \hline
$T_{2,q} \cup A$, $q$ even & $2+2|q|$ & 3 & $|q|$, $|q|$, 2 \\ \hline
$T_{2,3} \cup B$ & 18 & 2 & 12, 6  \\ \hline 
$L(1/2, 1/2, p/q; k)$, $q$ odd & $2(q+1)\vert (k-1)q - p\vert$ & 2 & \parbox{1in}{$2q\vert (k-1)q - p\vert$, \\$2\vert (k-1)q - p\vert$}  \\ \hline 
$L(1/2, 1/2, p/q; k)$, $q$ even & $2(q+1)\vert (k-1)q - p\vert$ & 3 & \parbox{1in}{$q\vert (k-1)q - p\vert$, \\$q\vert (k-1)q - p\vert$, \\$2\vert (k-1)q - p\vert$}  \\ \hline 
$L(k, p/q) \cup C$, $q$ odd & $2q(\vert kq - p\vert+1)$ & 2 & \parbox{1in}{$2q\vert kq - p\vert$, $2q$}  \\ \hline 
$L(k, p/q) \cup C$, $q$ even & $2q(\vert kq - p\vert+1)$ & 3 & \parbox{1.5in}{$q\vert kq - p\vert$, $q\vert kq - p\vert$, $2q$}  \\ \hline 
\end{tabular}
\end{center}
\caption{Orders of finite involutory quandles.}
\label{orders}
\end{table}

\end{document}